\documentclass[sn-mathphys]{mysn-jnl}
\usepackage{amsmath,amsthm,amsfonts,amssymb}

\jyear{2022}%

\theoremstyle{thmstyleone}%
\newtheorem{theorem}{Theorem}[section]
\newtheorem{corollary}[theorem]{Corollary}
\newtheorem{lemma}[theorem]{Lemma}

\begin{document}

\title[An upper bound for the minimum modulus in a squarefree covering]{An upper bound for the minimum modulus in a covering system with squarefree moduli}


\author[1]{\fnm{Maria} \sur{Cummings}}\email{mcummings780@gmail.com}

\author*[1]{\fnm{Michael} \sur{Filaseta}}\email{filaseta@math.sc.edu}

\author[1]{\fnm{Ognian} \sur{Trifonov}}\email{trifonov@math.sc.edu}

\affil*[1]{\orgdiv{Department of Mathematics}, \orgname{University of South Carolina}, \orgaddress{\street{1523 Greene Street}, \city{Columbia}, \state{SC}, \postcode{29208},  \country{USA}}}


\abstract{Based on work of P.~Balister, B.~Bollobás, R.~Morris, J.~Sahasrabudhe and M.~Tiba, we show that if a covering system has distinct squarefree moduli, then the minimum modulus is at most 118.  We also show that in general the \textit{k}-th smallest modulus in a covering system with distinct moduli (provided it is required for the covering) is bounded by an absolute constant.}




\maketitle

\section{Introduction}

A covering system (or a covering) is a finite system of congruences $x \equiv a_{j} \pmod{m_{j}}$, $j \in \{ 1, 2, \ldots, r \}$, such that every integer satisfies at least one of the congruences.
Of particular interest is the case in which all the moduli are distinct.  
In 1950, P.~Erd{\H o}s \cite{ref3} wrote, ``It seems likely that for every $c$ there exists such a system all the moduli of which are $> c$."
In other words, Erd{\H o}s felt that the minimum modulus in a covering system having distinct moduli can be arbitrarily large.
Establishing whether that is indeed the case has become known as the minimum modulus problem for covering systems,
and Erd{\H o}s offered \$1000 for a resolution to the problem \cite[Section F13]{guyversion3}.
The minimum modulus problem has only fairly recently been resolved by R.~Hough \cite{ref9} who showed, contrary to what Erd{\H o}s believed, the minimum
modulus is bounded and in particular $\le 10^{16}$. 
More recent work by P.~Balister, B.~Bollob{\'a}s, R.~Morris, J.~Sahasrabudhe and M.~Tiba \cite{ref2} has brought to light some new ideas which have led to 
a simpler argument producing an upper bound of $615999$ on the minimum modulus.

This paper concerns the related question of covering systems with distinct moduli that are all squarefree, that is each modulus is not divisible by the square of a prime.
In \cite{ref9} and \cite{ref1}, Hough and Balister, Bollob{\'a}s, Morris, Sahasrabudhe and Tiba, respectively, use the case of squarefree moduli to illustrate their approaches
to the minimum modulus problem.  In particular, the latter give a simple exposition of their \textit{distortion method} by showing that the minimum modulus in the case of
distinct squarefree moduli is bounded.  They end their paper by indicating that a direct application of what is written there gives a ``(fairly terrible) bound of roughly exp(10200) for the minimum modulus in a covering system" having distinct squarefree moduli.  As they note, finding a small bound in this case is not really the point, given the more interesting minimum modulus problem in which the moduli are not necessarily squarefree is resolved now with a bound of $615999$ on the minimum modulus.

Nevertheless, Erd{\H o}s pointed out that the congruences
\begin{align*}
    \begin{split}
    n &\equiv 0{\hskip -3pt}\pmod{2}\\
    n &\equiv 0 {\hskip -3pt}\pmod{3}\\
    n &\equiv 0 {\hskip -3pt}\pmod{5}\\
    n &\equiv 1 {\hskip -3pt}\pmod{6}\\
    n &\equiv 0 {\hskip -3pt}\pmod{7}
    \end{split}
    \begin{split}
    n &\equiv 1 {\hskip -3pt}\pmod{10}\\
    n &\equiv 1 {\hskip -3pt}\pmod{14}\\
    n &\equiv 2 {\hskip -3pt}\pmod{15}\\
    n &\equiv 2 {\hskip -3pt}\pmod{21} \\
    n &\equiv 23 {\hskip -3pt}\pmod{30}
    \end{split}
    \hspace{10mm}
    \begin{split}
    n &\equiv 4 {\hskip -3pt}\pmod{35}\\
    n &\equiv 5 {\hskip -3pt}\pmod{42} \\
    n &\equiv 59 {\hskip -3pt}\pmod{70}\\
    n &\equiv 104 {\hskip -3pt}\pmod{105}
    \end{split}
\end{align*}
provide an example of a covering in which the moduli are distinct squarefree numbers, and it is unknown as to whether there is such an example where
the minimum modulus is $> 2$.  This problem, due to J.~Selfridge, goes back to at least 1981, being cited with the example above in \cite[Section F13]{guyversion1}.  Thus, the
problem of obtaining a good upper bound on the minimum modulus problem in the case of distinct squarefree moduli is an interesting one.

In this paper, we make the following progress on this question.

\begin{theorem}\label{thm118}
Every covering system with distinct squarefree moduli has a minimum modulus which is $\le 118$.
\end{theorem}

\noindent
Our arguments are not novel in that we basically take the eloquent exposition given by Balister, Bollob{\'a}s, Morris, Sahasrabudhe and Tiba in \cite{ref1} and 
simply refine the arguments to produce our upper bound.  

There is a natural question of whether a bound can be given on the second, third, etc., smallest modulus of a covering system with distinct (possibly squarefree) moduli. 
We give a simple elementary argument that the following is a consequence of R.~Hough's initial work on this topic \cite{ref9}.

\begin{theorem}\label{secondresult}
Fix a non-negative integer $k$.  Then there exists a $B(k+1)$
satisfying the following.
Let $C$ be a covering with
moduli $m_{1} < m_{2} < \cdots < m_{r}$ and congruences
\begin{equation}\label{congrs}
x \equiv a_{j} \pmod{m_{j}}, \quad 1 \le j \le r,
\end{equation}
satisfying $r \ge k+1$ and the first $k$ congruencies in \eqref{congrs} do not
form a covering of the integers.  Then $m_{k+1} \le B(k+1)$.
\end{theorem}

This leads to some further questions.  
Our proof can be adjusted easily to be constructive, though we will make no attempt to find explicit bounds on $B(k)$ as our approach is undoubtedly not optimal.  
How small can one make such bounds?  In particular, what is the smallest possible value of $B(k)$?  Theorem~\ref{secondresult} implies the same result holds in the case that the
moduli $m_{j}$, $1 \le j \le k$, are all squarefree.  What are the answers to these questions in the case of squarefree moduli?  Knowing the various moduli in a covering
system are bounded as in Theorem~\ref{secondresult} but taking into account that the number of moduli in a covering system can be arbitrarily large, is it possible to give a complete
classification of all possible covering systems involving distinct moduli?  A similar question can be asked in the case that the moduli are not distinct, though this will necessarily be a larger classification.

Before leaving this introduction, we mention some related literature.
Some other recent work on this subject include
\cite{balisteretal,balisteretal2,wilson,hough2,nielsen,owens}.
More information related to covering systems with distinct squarefree moduli can be found in Kruckenberg's dissertation \cite{krukenberg}.
For example, he shows that the only covering systems with distinct squarefree moduli  
where the least common multiple $L$ of the moduli is the product of $4$ or fewer primes is when $L = 210$
and that there is a covering system with distinct squarefree moduli which uses the modulus $2$ but not the modulus $3$.  
A number of applications of covering systems beyond the paper by Erd{\H o}s \cite{ref3} can be found in the references given in \cite{juillerat}.


\section{Preliminary background}\label{sec2}

In this section, as well as the next, our approach is based on the work of P. Balister, B. Bollobás, R. Morris, J. Sahasrabudhe, and M. Tiba in \cite{ref1}
where they give the basic idea behind their method in \cite{ref2} but restrict to the case where the moduli are distinct squarefree numbers.  
Our goal is simply to refine the ideas to allow for the bound $118$ in Theorem~\ref{thm118}.  
However, we will make use of weighted sums instead of probabilities to give a slightly different perspective without any real change in content.  

We begin with the following.

\begin{lemma}\label{covcheck}
Let $\mathcal{C}$ be a system of congruences consisting of moduli $m_1,\ldots, m_r$, and set $L=\text{lcm}(m_1,\ldots, m_r)$.
Then, $\mathcal{C}$ is a covering system if and only if every integer in $[1,L]$ is satisfied by a congruence in $\mathcal{C}$.
\end{lemma}

For example, consider the covering: 
\begin{align*}
\begin{split}
n &\equiv 0 {\hskip -3pt}\pmod{2}\\
n &\equiv 0 {\hskip -3pt}\pmod{3}\\
n &\equiv 1 {\hskip -3pt}\pmod{4}
\end{split}
\begin{split}
n &\equiv 3 {\hskip -3pt}\pmod{8}\\
n &\equiv 7 {\hskip -3pt}\pmod{12}\\
n &\equiv 23 {\hskip -3pt}\pmod{24}.
\end{split}
\end{align*}
Here $L=24$.
Lemma~\ref{covcheck} implies that we only need to check that each integer from $1$ to $24$ 
satisfies at least one congruences to verify that the above is a covering system, and this can easily be done.

\begin{proof}[Proof of Lemma \ref{covcheck}]
Let $\mathcal{C}$ be a system of congruences with the least common multiple of the moduli equal to $L$.
If $\mathcal{C}$ is a covering system, then every integer in $[1,L]$ satisfies a congruence in $\mathcal{C}$.
Now, suppose every integer in $[1,L]$ satisfies a congruence in $\mathcal{C}$.
Let $n$ be an integer.
Let $a\equiv n\pmod{L}$ where $1\le a\le L$.
Then, there exists a congruence $x\equiv b\pmod{m}$ in $\mathcal{C}$ such that $a\equiv b\pmod{m}$.
Since $a\equiv n\pmod{L}$ and $m$ divides $L$, we have $n\equiv a\equiv b\pmod{m}$.
Therefore, $n$ satisfies the congruence $x\equiv b\pmod{m}$ in $\mathcal{C}$.
Thus, $\mathcal{C}$ is a covering system, and the lemma follows.
\end{proof}

We return to the example given in the introduction involving squarefree moduli, where the least common multiple of the moduli is $L = 210 = 2 \cdot 3 \cdot 5 \cdot 7$.  
One can use Lemma~\ref{covcheck} to quickly verify that the congruences given there form a covering system of the integers. 
We can also think of covering the associated elements of $Q=S_1\times S_2\times S_3\times S_4$ where $S_1=\{1,2\}$, $S_2=\{1,2,3\}$, $S_3=\{1,2,3,4,5\}$, and $S_4=\{1,2,3,4,5,6,7\}$.
With this approach each congruence covers a portion of $Q$.
For example, the congruence $n \equiv 0 \pmod{2}$ covers $\{2\}\times S_2\times S_3\times S_4 \subseteq Q$.
By the Chinese Remainder Theorem, each integer in $[1,L]$ uniquely corresponds to an element of $Q$.
For example, the integer $77$ corresponds to $(1,2,2,7)$ in $Q$ since $77\equiv 1 \pmod{2}$, $77\equiv 2 \pmod{3}$, $77\equiv 2  \pmod{5}$, and $77\equiv 7\pmod{7}$.
Given Lemma~\ref{covcheck}, for distinct squarefree moduli, we can view a covering system 
as a system of congruences with the product of the moduli $L$ for which each element of $Q$ corresponds to an integer that satisfies at least one of the congruences.

In general, we let $S_1,\ldots, S_n$ be finite sets.
We define a \textit{hyperplane} to be $A=Y_1\times\cdots\times Y_n$ where $Y_j\subseteq S_j$ and $\lvert Y_j\rvert\in\{1,\lvert S_j \rvert\}$ for $j\in\{1,\ldots,n\}$.
We also define two hyperplanes $A$ and $A'$ to be \textit{parallel} if $F(A)=F(A')$ where $F(A)=\{j:\lvert Y_j \rvert=1\}$.
We call $F(A)$ the set of fixed coordinates of $A$.
We consider the set of natural numbers $\mathbb{N}$ to be the set of positive integers.

\begin{theorem}[{P. Balister, B. Bollobás, R. Morris, J. Sahasrabudhe, and M. Tiba \cite{ref1}}]\label{theorem1}
For every sequence of finite sets $S_1,S_2,\ \ldots\ $, each of size at least $2$, satisfying 
$\liminf_{k\rightarrow\infty} \lvert{S_k}\rvert/k>3,$ there is a positive integer $C$ such that the following holds.
Let $\mathcal{A}$ be a collection of hyperplanes that cover $Q=S_1\times\ldots\times S_n$ for some $n\in\mathbb{N}$, 
(that is, every element of $Q$ is on some hyperplane in $\mathcal{A}$).
Suppose no two hyperplanes in $\mathcal{A}$ are parallel.
Then, there exists a hyperplane $A\in\mathcal{A}$ with $F(A)\subseteq \{1,\ldots,C\}$.
\end{theorem}

The above result is the main result in \cite{ref1}.  To connect it to our earlier discussion,
let $\mathcal{C}$ be a covering system with distinct squarefree moduli. 
Let $p_j$ denote the $j^{\rm th}$ prime, and set $S_j=\{1,...,p_j\}$.
Observe that each $S_j$ is of size at least $2$, and we have $\liminf_{k\rightarrow\infty} \lvert S_k \rvert/k=\lim_{k\rightarrow\infty} p_k/k=\infty >3$.
Let $C$ be a positive integer as in Theorem~\ref{theorem1}.  
Let $p_n$ be the largest prime dividing a modulus in $\mathcal{C}$.
Set $Q=S_1\times\cdots\times S_n$.
Each congruence $x\equiv a\pmod{m}$ in $\mathcal{C}$ corresponds to a hyperplane 
$A_m=Y_1\times\cdots\times Y_n\subseteq Q$ where (i) if $p_j$ divides $m$, then $Y_j=\{b\}$ with $b\equiv a \pmod{p_j}$ and $b\in S_j$, 
and (ii) if $p_j$ does not divide $m$, then $Y_j=S_j$.
Then, the covering $\mathcal{C}$ corresponds to a finite collection of hyperplanes $\mathcal{A}$ which covers $Q$.
Note that the moduli of $\mathcal{C}$ are distinct, so the hyperplanes in $\mathcal{A}$ are pairwise non-parallel.
Thus, by Theorem~\ref{theorem1}, there exists an  $A_m\in\mathcal{A}$ with $F(A_m)\subseteq \{1,\ldots,C\}$.
Observe that this $m$ divides $p_1\cdots p_C$.
Therefore, we can obtain that the minimum modulus must be $\le p_1\cdots p_C$.

Our goal is to slightly alter the proof of Theorem~\ref{theorem1} in \cite{ref1} to give our proof of Theorem~\ref{thm118}.


\section{Further background}\label{sec3}

We will prove Theorem \ref{theorem1} by expanding upon the method outlined in \cite{ref1}.
Let $S_1,S_2,\ldots$ be an infinite sequence of finite sets.
As discussed at the end of the previous section, we will be connecting these sets to our covering system in Theorem~\ref{thm118} by taking
$S_j=\{1,...,p_j\}$, for each $j$, where $p_{j}$ denotes the $j^{\rm th}$ prime. 
For a positive integer $k$, define $Q_k=S_1\times \ldots \times S_k$.
Fix a positive integer $n$.
Let $\mathcal{A}$ be a collection of hyperplanes, pairwise non-parallel, that cover $Q_n$.
We define a \textit{weight} on a set $X=\{x_1,\ldots,x_k\}$ to be a function mapping each $x_i$ to $q_i\geq 0$ for $1\le i\le k$ such that $q_1+\ldots +q_k=1$. 
As we are setting the sum of the weights $q_{i}$ equal to $1$, these weights can be viewed as a probability assigned to the elements of $X$, as done in \cite{ref1}. 

Let $Q=Q_n$.
We define weights $w_n(x)$ on the elements $x$ of $Q$.
From the previous section, for covering systems, these weights correspond to weights on the integers in the interval $[1,L]$ where $L=p_1\cdots p_n$.
The weight of a subset $T\subseteq Q$ is defined as the sum of the weights of the elements in $T$, so $w_n(T)=\sum_{x\in T} w_n(x)$.
We interpret this to mean $w_n(\emptyset)=0$.
If $x=(a_1,\ldots,a_{k-1})\in S_1\times \ldots\times S_{k-1}$ and $y\in S_k$, then we write $w_k(x,y)=w_k((a_1,\ldots,a_{k-1},y))=w_n(A)$ where $A$ is the hyperplane 
\begin{align}\label{generalhyperplane}
A=\{a_1\}\times \{a_2\}\times\ldots\times \{a_{k-1}\}\times \{y\}\times S_{k+1}\times \ldots \times S_n.
\end{align}
In general, if $X\subseteq S_1\times \ldots\times S_k$, then we identify $w_k(X)$ with $w_n(X\times S_{k+1}\times \ldots \times S_n)$.
The fiber $F_x$ associated to $x=(a_1,\ldots,a_{k-1})\in S_1\times \ldots\times S_{k-1}$ is the set of tuples $(a_1,\ldots,a_{k-1},y)$ with $y \in S_{k}$.
At the $k$th stage, we will determine the weights of the hyperplanes in the form of \eqref{generalhyperplane}.
We define 
\[ \mathcal{A}_k=\{A\in \mathcal{A}:\max(F(A))=k\}. \]
In the languages of congruences, $\mathcal{A}_1$ corresponds to the set of congruences modulo $p_1=2$, $\mathcal{A}_2$ corresponds to the set of congruences modulo $p_2=3$ and $p_1p_2 = 6$, and so on.
We also define 
\[ B_k=\bigcup_{A\in\mathcal{A}_k} A. \]
With regard to coverings, $B_k$ corresponds to the elements of $Q_n$ which are covered by a congruence with a modulus whose largest prime divisor is $p_k$.
Note that if $A\in\mathcal{A}_k$, then $F(A)\subseteq \{1,\ldots,k\}$, so $B_k$ can also be thought of as a subset of $Q_k=S_1\times \ldots\times S_k$.

In our proof of Theorem \ref{theorem1}, by assigning weights to elements of $Q$ in the manner below and supposing $F(A)\not\subseteq\{1,\ldots,C\}$ for every hyperplane $A\in\mathcal{A}$, we prove that the collection of hyperplanes $\mathcal{A}$ does not cover $Q$ to obtain our result by contradiction.

For each $k$, we will choose $\delta_k\in[0,1/2]$.
We define weights $w_k$ inductively as follows.
As noted above, we can view $B_k$ as a subset of $Q_k=S_1\times \ldots\times S_k$, and do so.
When $k=1$, if $y\in S_1$ and $\lvert B_1\rvert /\lvert S_1\rvert \le \delta_1$, we set
\begin{align*}
w_1(y)=\begin{cases}
        0\quad &\text{if} \, y\in B_1\\
        \dfrac{1}{\lvert S_1\rvert -\lvert B_1\rvert } \quad &\text{if} \, y \not\in B_1.
        \end{cases}
\end{align*}
If $y\in S_1$ and $\lvert B_1\rvert /\lvert S_1\rvert > \delta_1$, we set
\begin{align*}
w_1(y)=\begin{cases}
          \dfrac{(\lvert B_1\rvert /\lvert S_1\rvert )-\delta_1}{(\lvert B_1\rvert /\lvert S_1\rvert )(1-\delta_1)}\cdot \dfrac{1}{\lvert S_1\rvert } \quad &\text{if} \, y \in B_1 \\[8pt]
          \dfrac{1}{1-\delta_1}\cdot \dfrac{1}{\lvert S_1\rvert } \quad &\text{if} \, y \not\in B_1. \\
     \end{cases}
\end{align*}
Observe that in both cases, we have $\sum_{y\in S_1}w_1(y) = 1$; for example, if $\lvert B_1\rvert /\lvert S_1\rvert > \delta_1$, then
\begin{align*}
\sum_{y\in S_1}w_1(y) &= \sum_{\substack{y\in S_1 \\ y \in B_1}} w_1(y) + \sum_{\substack{y\in S_1 \\ y \not\in B_1}} w_1(y) \\
&= \sum_{\substack{y \in B_1}} \dfrac{(\lvert B_1\rvert /\lvert S_1\rvert )-\delta_1}{(\lvert B_1\rvert /\lvert S_1\rvert )(1-\delta_1)}\cdot \dfrac{1}{\lvert S_1\rvert } 
 + \sum_{\substack{y\in S_1 \\ y \not\in B_1}}  \dfrac{1}{1-\delta_1}\cdot \dfrac{1}{\lvert S_1\rvert } \\
&= \dfrac{(\lvert B_1\rvert /\lvert S_1\rvert )-\delta_1}{(\lvert B_1\rvert /\lvert S_1\rvert )(1-\delta_1)}\cdot \dfrac{\lvert B_1\rvert}{\lvert S_1\rvert } 
 +  \dfrac{1}{1-\delta_1}\cdot \dfrac{\lvert S_1\rvert - \lvert B_1\rvert}{\lvert S_1\rvert } 
 = 1.
 \end{align*}
The above weights correspond to setting $k=1$ and replacing $\alpha_{1}(x)$ with $\lvert B_1\rvert /\lvert S_1\rvert $ and $w_0(x)$ with $1$ in the discussion below. 

Suppose $k\geq 2$ and $w_{k-1}$ is defined on $Q_{k-1}$.
For each $x\in Q_{k-1}$, we define
\begin{align*}
  \alpha_k(x)=\frac{\lvert \{y\in S_k:(x,y)\in B_k\}\rvert }{\lvert S_k\rvert }=\frac{\lvert F_x\cap B_k\rvert }{\lvert S_k\rvert },  
\end{align*}
which is the proportion of the fiber $F_x=\{(x,y):y\in S_k\}$ that is covered by one or more hyperplanes in $\mathcal{A}_k$.
If $\alpha_k(x)\le \delta_k$, we set
\begin{align*}
w_k(x,y)=\begin{cases}
          0 \quad &\text{if} \, (x,y) \in B_k \\[8pt]
          \dfrac{1}{1-\alpha_k(x)}\cdot \dfrac{w_{k-1}(x)}{\lvert S_k\rvert } \quad &\text{if} \, (x,y) \not\in B_k. \\
     \end{cases}
\end{align*}
If $\alpha_k(x)> \delta_k$, we set 
\begin{align*}
w_k(x,y)=\begin{cases}
          \dfrac{\alpha_k(x)-\delta_k}{\alpha_k(x)(1-\delta_k)}\cdot \dfrac{w_{k-1}(x)}{\lvert S_k\rvert } \quad &\text{if} \, (x,y) \in B_k \\[8pt]
          \dfrac{1}{1-\delta_k}\cdot \dfrac{w_{k-1}(x)}{\lvert S_k\rvert } \quad &\text{if} \, (x,y) \not\in B_k. \\
     \end{cases}
\end{align*}

In both the cases $\alpha_k(x)\le \delta_{k}$ and $\alpha_k(x) > \delta_{k}$, we justify that $\sum_{y\in S_k} w_k(x,y)=w_{k-1}(x)$, 
so weight is preserved along the fibers with each increase of $k$.
If $x$ is an element of $Q_{k-1}$ and $\alpha_k(x)\le \delta_{k}$, we have
\begin{align*}
\sum_{y\in S_k}w_k(x,y)&=\sum\limits_{\substack{y\in S_k \\ (x,y)\in B_k}} w_k(x,y)+\sum\limits_{\substack{y\in S_k \\ (x,y)\not\in B_k}} w_k(x,y) \\
&= \sum\limits_{\substack{y\in S_k \\ (x,y)\in B_k}} 0+\sum\limits_{\substack{y\in S_k \\ (x,y)\not\in B_k}} \frac{1}{1-\alpha_k(x)}\cdot\frac{w_{k-1}(x)}{\lvert S_k\rvert } \\
&= \frac{1}{1-\alpha_k(x)}\cdot\frac{w_{k-1}(x)}{\lvert S_k\rvert }\sum\limits_{\substack{y\in S_k \\ (x,y)\not\in B_k}} 1 \\
&=\frac{1}{1-\alpha_k(x)}\cdot\frac{w_{k-1}(x)}{\lvert S_k\rvert }\cdot (\lvert S_k\rvert -\lvert S_k\rvert \alpha_k(x))\\[5pt]
&=w_{k-1}(x),
\end{align*}
where we have used the definition of $\alpha_{k}(x)$ in the second from the last equality. 
Also, if $x$ is an element of $Q_{k-1}$ and $\alpha_k(x)>\delta_{k}$, we then have
\begin{align*}
\sum_{y\in S_k}w_k(x,y)&=\sum\limits_{\substack{y\in S_k \\ (x,y)\in B_k}} w_k(x,y)+\sum\limits_{\substack{y\in S_k \\ (x,y)\not\in B_k}} w_k(x,y) \\
&= \sum\limits_{\substack{y\in S_k \\ (x,y)\in B_k}} \frac{\alpha_k(x)-\delta_{k}}{\alpha_k(x)(1-\delta_{k})}\cdot\frac{w_{k-1}(x)}{\lvert S_k\rvert }+\sum\limits_{\substack{y\in S_k \\ (x,y)\not\in B_k}} \frac{w_{k-1}(x)}{\lvert S_k\rvert (1-\delta_{k})}  \\
&= \frac{\alpha_k(x)-\delta_{k}}{\alpha_k(x)(1-\delta_{k})}\cdot\frac{w_{k-1}(x)}{\lvert S_k\rvert }\sum\limits_{\substack{y\in S_k \\ (x,y)\in B_k}} 1+ \frac{w_{k-1}(x)}{\lvert S_k\rvert (1-\delta_{k})}\sum\limits_{\substack{y\in S_k \\ (x,y)\not\in B_k}}1  \\
&=\frac{\alpha_k(x)-\delta_{k}}{\alpha_k(x)(1-\delta_{k})}\cdot\frac{w_{k-1}(x)}{\lvert S_k\rvert }\cdot \alpha_k(x)\lvert S_k\rvert \\
&\qquad \qquad \qquad +\frac{w_{k-1}(x)}{\lvert S_k\rvert (1-\delta_{k})}(\lvert S_k\rvert -\lvert S_k\rvert \alpha_k(x)) \\[5pt]
&=\frac{(\alpha_k(x)-\delta_{k})w_{k-1}(x)}{1-\delta_{k}}+\frac{(1-\alpha_k(x))w_{k-1}(x)}{1-\delta_{k}}\\[5pt]
&=w_{k-1}(x).
\end{align*}
Thus, weight is preserved along the fibers with each increase of $k$, so that in particular we have $\sum_{x \in Q_k}w_k(x) = 1$ since as already noted the equation holds for $k = 1$.
In other words, as we extend our definition of the weights from $w_{k-1}(x)$ for $x \in Q_{k-1}$ 
to $w_{k}(x)$ for $x \in Q_{k}$, we maintain the property that the sum of all the weights is $1$.  
With $\mathcal A$ and $Q = Q_{n}$ as in Theorem~\ref{theorem1}, note that
each hyperplane $A \in \mathcal A$ belongs to exactly one set $B_{k}$ for $1 \le k \le n$.  
Thus, the following holds.

\begin{lemma}\label{newlemma1}
Let $\mathcal{A}$ be a collection of hyperplanes in $Q=S_1\times\ldots\times S_n$.
If 
\[
\sum_{k=1}^n w_k(B_k) < 1,
\]
then $\mathcal{A}$ does not cover $Q$.
\end{lemma}

The basic idea therefore is to show the inequality in Lemma~\ref{newlemma1} 
when $\mathcal{A}$ comes from a set of congruences with distinct squarefree moduli $> 118$.  
This idea describes the basic approach of the authors in \cite{ref1, ref2} as well.


\section{Upper bounds on $w_{k}(B_{k})$}\label{sec3}

For any element $x\in Q_{k-1}$ and any element $y\in S_k$, we justify that 
\begin{equation}\label{upperbduno}
  w_k(x,y)\le \frac{1}{1-\delta_k}\cdot\frac{w_{k-1}(x)}{\lvert S_{k} \rvert } \qquad \text{for } k \ge 2.
\end{equation} 
In the case that $k = 1$, by considering $\lvert B_{1} \rvert/\lvert S_{1} \rvert \le \delta_{1}$ and $\lvert B_{1} \rvert/\lvert S_{1} \rvert > \delta_{1}$ 
separately, similar to the argument which follows for $k \ge 2$, one can easily verify \eqref{upperbduno} with $w_{0}(x)$ replaced by $1$.  
For $k \ge 2$ and $\alpha_k(x)\le \delta_k$, we have 
\begin{align*}
  w_k(x,y)\le \frac{1}{1-\alpha_k(x)}\cdot \frac{w_{k-1}(x)}{\lvert S_{k} \rvert }\le   \frac{1}{1-\delta_k}\cdot \frac{w_{k-1}(x)}{\lvert S_{k} \rvert }.
\end{align*}
If $k \ge 2$ and $\alpha_k(x)>\delta_k$ and $(x,y)\not\in B_k$, our result holds by the definition of $w_k(x,y)$.
If $k \ge 2$ and $\alpha_k(x)>\delta_k$ and $(x,y)\in B_k$, then we obtain
\begin{align*}
    w_k(x,y)&\le \frac{\alpha_k(x)-\delta_k}{\alpha_k(x)(1-\delta_k)}\cdot \frac{w_{k-1}(x)}{\lvert S_{k} \rvert }\\[5pt]
    &= \left( \frac{1}{1-\delta_k}-\frac{\delta_k}{\alpha_k(x)(1-\delta_k)}\right)\cdot \frac{w_{k-1}(x)}{\lvert S_{k} \rvert }\\[5pt]
    &\le\frac{1}{1-\delta_k}\cdot \frac{w_{k-1}(x)}{\lvert S_{k} \rvert }.
\end{align*}
Thus, \eqref{upperbduno} holds.

For a hyperplane $A=Y_1\times \ldots \times Y_n$ and a set $U\subseteq\{1,\ldots,n\}$, 
we define $A^U=Y_1^U\times \ldots \times Y_n^U$ to be the hyperplane with $Y_i^U=Y_i$ if $i\in U$ and $Y_i^U=S_i$ if $i\not\in U$.
We set $A'=A^{\{1,\ldots,k-1\}}$.
For each $J\subseteq\{1,\ldots,n\}$, we define 
\[
\nu(J)=\prod_{j\in J} \dfrac{1}{(1-\delta_j)\lvert S_{j} \rvert }
\qquad \text{and} \qquad
 \Vert J \Vert=\prod_{j\in J} \lvert S_{j} \rvert .
\]
To clarify, we set, as usual, empty products to be $1$ so that $\Vert \emptyset \Vert = 1$.

If $A$ is a hyperplane corresponding to some congruence with squarefree modulus $m$ in a covering system, then $\Vert F(A) \Vert = m$.  
We are interested in showing that the modulus of some congruence is bounded above by $C_0 = 118$.  
We use $C_{0}$ instead of $118$ for the moment to clarify that most of what is done below is independent of the value of $C_{0}$,
so the reader can view this for the time being as a variable to be determined.  
We assume 
\begin{equation}\label{normbound}
\Vert F(A) \Vert > C_0
\qquad \text{for all }
A \in \mathcal A,
\end{equation}
with a goal of obtaining a contradiction.


Our next estimate will help us formulate a bound on $w_k(B_k)$ which we will use when $k$ is small.

\begin{lemma}\label{sflemma1}
Let $\mathcal{A}$ be a collection of hyperplanes, pairwise non-parallel.
Then, for $k\geq 1$, we have
\begin{align*} 
w_k(B_k)\le \sum_{A\in\mathcal{A}_k} w_k(A)\le \sum_{A\in\mathcal{A}_k}\nu(F(A))= \sum_{A\in\mathcal{A}_k} \prod_{j\in F(A)} \frac{1}{(1-\delta_j)\lvert S_{j} \rvert }.
\end{align*}
\end{lemma}

\begin{proof}
Since $B_k=\bigcup_{A\in\mathcal{A}_k} A$, we have $w_k(B_k)\le \sum_{A\in\mathcal{A}_k} w_k(A)$.
We will induct on $k$ to prove $w_k(A)\le \nu(F(A))$ for $A\in\mathcal{A}_k$.
For the induction, we will want more generally to look at hyperplanes not necessarily in $\mathcal A$ as well.
For this reason, we denote by $\mathcal A^{\text{all}}_{k}$ the set of all hyperplanes $A$ in $S_{1} \times S_{2} \times \cdots \times S_{k}$
(or equivalently in $S_{1} \times S_{2} \times \cdots \times S_{n}$)
for which $\max(F(A)) = k$.   We justify by induction that 
\begin{equation}\label{couldbealemma}
w_k(A) \le \nu(F(A)) = \prod_{j\in F(A)} \frac{1}{(1-\delta_j)\lvert S_{j} \rvert }
\qquad \text{for all } A \in \mathcal A^{\text{all}}_{k}.
\end{equation}
For our base case, consider $k=1$. 
With $k = 1$ and $A \in \mathcal{A}^{\text{all}}_k$, we see that $F(A)= \{1\}$. 
Since $F(A) = \{1\}$, we obtain $A = \{y\}$ (or equivalently $A = \{ y \} \times S_{2} \times \cdots \times S_{n}$) for some $y \in S_1$. 
The comment about the case $k=1$ after \eqref{upperbduno} now implies
\[
w_1(A)=w_1(y) \le \frac{1}{1-\delta_1}\cdot \frac{1}{\lvert S_{1} \rvert } = \nu(\{1\}).
\]
Thus, \eqref{couldbealemma} holds when we restrict to $A \in \mathcal A^{\text{all}}_{1}$.

For our inductive step, suppose that for some $k \in \{ 2, \ldots, n \}$, we have 
$w_{j}(A) \le \nu(F(A))$ whenever $A \in \mathcal A^{\text{all}}_{j}$, where $1 \le j < k$.
Let $A \in \mathcal A^{\text{all}}_{k}$.  
As before, we have $k\in F(A)$.
With $A'=A^{\{1,\ldots,k-1\}}$,
we obtain from \eqref{upperbduno} that
\begin{align*}
w_k(A) \le \frac{1}{(1-\delta_k)\lvert S_{k} \rvert }\cdot w_{k-1}(A').
\end{align*}
Since $A'\subseteq\{1,\ldots,k-1\}$ and $F(A')=F(A)\setminus \{k\}$, then by our inductive hypothesis, we have $w_{k-1}(A')\le \nu(F(A)\setminus\{k\})$.
Thus, for $k\geq 1$, we have
\[
w_k(A) \le\frac{1}{(1-\delta_k)\lvert S_{k} \rvert }\cdot w_{k-1}(A')
\le \frac{1}{(1-\delta_k)\lvert S_{k} \rvert }\cdot \nu(F(A)\setminus\{k\})
=\nu(F(A)).
\]
This completes the induction argument.

We are now able to conclude
\begin{align*}
w_k(B_k)\le \sum_{A\in\mathcal{A}_k} w_k(A)\le \sum_{A\in\mathcal{A}_k}\nu(F(A)) = \sum_{A\in\mathcal{A}_k} \prod_{j\in F(A)} \frac{1}{(1-\delta_j)\lvert S_{j} \rvert }  ,
\end{align*}
which completes our proof.
\end{proof}

\begin{corollary}\label{corsflemma1}
Let $\mathcal{A}$ be a collection of hyperplanes, pairwise non-parallel,
satisfying \eqref{normbound}. 
Then
\begin{align*} 
w_k(B_k)&\le \frac{1}{(1-\delta_k)\lvert S_{k} \rvert } \sum_{\substack{J \subseteq \{1,\ldots,k-1\} \\ \Vert J \Vert > C_0/\lvert S_{k} \rvert }} \nu(J) \\
&= \frac{1}{(1-\delta_k)\lvert S_{k} \rvert } \sum_{\substack{J \subseteq \{1,\ldots,k-1\} \\ \Vert J \Vert > C_0/\lvert S_{k} \rvert }} \prod_{j\in J} \frac{1}{(1-\delta_j)\lvert S_{j} \rvert }.
\end{align*}
\end{corollary}

\begin{proof}
For $A \in \mathcal{A}_k$, we write
\[
F(A) = J \cup \{ k \},
\]
where $J \subseteq \{1,\ldots, k-1\}$.  For such $A$ and $J$, we have
\[
\Vert F(A) \Vert = \Vert J \Vert \cdot \lvert S_{k} \rvert .
\]
In particular, $\Vert F(A) \Vert > C_0$ is equivalent to $\Vert J \Vert > C_0/\lvert S_{k} \rvert $.  
Also, since the hyperplanes in $\mathcal A$ are pairwise non-parallel, different $A \in \mathcal{A}_k$ correspond to different $J \subseteq \{1,\ldots, k-1\}$.  
Since every $A \in \mathcal{A}_k$ satisfies \eqref{normbound}, the result follows from Lemma~\ref{sflemma1}.
\end{proof}

For our next bound, we define the weighted sum 
\[
E_{k-1} =
\begin{cases}
        \displaystyle \sum_{x\in Q_{k-1}} \alpha_k(x)^2w_{k-1}(x)\quad &\text{if} \, k\geq 2\\[10pt]
        (\lvert B_1\rvert/\lvert S_1\rvert )^2 \quad &\text{if} \, k=1.
 \end{cases}
\]
This weighted sum can be viewed as the expected value of $\alpha_k(x)^2$ and, for this reason, was denoted $\mathbb E_{k-1}[\alpha_{k}(x)^{2}]$ in \cite{ref1}.
We will only treat the weighted sum through the definition above.
We state our next result for $k \ge 1$, but note that the result and separate argument for $k = 1$ is not needed in the rest of the paper.

\begin{lemma}\label{sflemma2}
Let $\mathcal{A}$ be a collection of hyperplanes in $Q=S_1\times\ldots\times S_n$.
Let $k\geq 1$, and suppose $\delta_{k} \in (0,1/2]$ (i.e., $\delta_{k} \ne 0$).  
Then
\begin{align*}
w_k(B_k) \le \frac{1}{4\delta_k(1-\delta_k)} E_{k-1}.
\end{align*}
\end{lemma}

\begin{proof}
First, consider $w_k(B_k)$ with $k\geq 2$.
We have 
\begin{align*}
w_k(B_k)& =\sum_{x\in Q_{k-1}} \sum\limits_{\substack{y\in S_k \\ (x,y)\in B_k}} w_k(x,y)\\
&\le\sum_{x\in Q_{k-1}} \lvert F_x\cap B_k \rvert \cdot \max\bigg\{0,\dfrac{\alpha_k(x)-\delta_k}{\alpha_k(x)(1-\delta_k)}\bigg\}\cdot \dfrac{w_{k-1}(x)}{\lvert S_{k} \rvert }.
\end{align*}
Since $\alpha_k(x)=\lvert F_x\cap B_k \rvert /\lvert S_{k} \rvert $, we obtain
\begin{align*}
w_k(B_k)\le\frac{1}{1-\delta_k} \sum_{x\in Q_{k-1}}\max\{0,\alpha_k(x)-\delta_k\}\cdot w_{k-1}(x).
\end{align*}
An important observation from \cite{ref1} is that 
\[
4\delta_k^2-4\delta_k\alpha_k(x)+\alpha_k(x)^2=(2\delta_k-\alpha_k(x))^2\geq 0,
\] 
so $\alpha_{k}(x)^2/4\delta_k\geq \alpha_k(x)-\delta_k$.
Thus,
\begin{align*}
w_k(B_k)& \le \frac{1}{1-\delta_k} \sum_{x\in Q_{k-1}}\frac{\alpha_k(x)^2}{4\delta_k}\cdot w_{k-1}(x)\\
&=\frac{1}{4\delta_k(1-\delta_k)} \sum_{x\in Q_{k-1}}\alpha_k(x)^2\cdot w_{k-1}(x)\\
&=\frac{1}{4\delta_k(1-\delta_k)} E_{k-1}.
\end{align*}
In the case that $k=1$, we have 
\begin{align*}
w_1(B_1)& =\sum\limits_{y\in B_1} w_1(y)\le \lvert B_1 \rvert \cdot \max\bigg\{0,\dfrac{(\lvert B_1 \rvert /\lvert S_{1} \rvert )-\delta_1}{(\lvert B_1 \rvert /\lvert S_{1} \rvert )(1-\delta_1)}\bigg\}\cdot \dfrac{1}{\lvert S_{1} \rvert }.
\end{align*}
Following the arguments above with $\alpha_k(x)$ replaced by $\lvert B_1 \rvert /\lvert S_{1} \rvert$, we obtain
\begin{align*}
    w_1(B_1)\le \frac{1}{4\delta_1(1-\delta_1)} \left(\frac{\lvert B_1 \rvert }{\lvert S_{1} \rvert }\right)^2=\frac{1}{4\delta_1(1-\delta_1)} E_{0}.
\end{align*}
The lemma follows.
\end{proof}


\begin{lemma}\label{sflemma3}
Let $\mathcal{A}$ be a collection of hyperplanes, pairwise non-parallel, in $Q$ satisfying \eqref{normbound} for some constant $C_0\geq 0$.
Then, for each integer $k\in[1, n]$, we have 
\begin{align*}
E_{k-1} \le \frac{1}{\lvert S_{k} \rvert ^2} \sum_{\substack{F_1,F_2\subseteq \{1,\ldots,k-1\} \\ \Vert F_1 \Vert>C_0/\lvert S_{k} \rvert ,\  \Vert F_2 \Vert>C_0/\lvert S_{k} \rvert }} \ \prod_{j\in F_1\cup F_2} \frac{1}{(1-\delta_j)\lvert S_{j} \rvert }.
\end{align*}
\end{lemma}

\begin{proof}
Similar to the proof of Corollary~\ref{corsflemma1}, for
$A$ in $\mathcal A_{k}$, we write
\[
F(A) = F_0(A) \cup \{ k \}
\]
for some $F_0(A)$ in $\{1,\ldots, k-1\}$.  Then
\[
\Vert F(A) \Vert = \Vert F_0(A) \Vert \cdot \lvert S_{k} \rvert.
\]
Observe that the condition $\Vert F(A) \Vert > C_0$ in \eqref{normbound} is equivalent to $\Vert F_0(A) \Vert > C_0/\lvert S_{k} \rvert $.

From the definition of $\alpha_k(x)$, we obtain
\begin{align*}
\alpha_k(x)=\frac{1}{\lvert S_{k} \rvert }\sum_{\substack{y\in S_k\\(x,y)\in B_k}} 1\le \frac{1}{\lvert S_{k} \rvert }\sum_{y\in S_k} \ \sum_{\substack{A\in\mathcal{A}_k \\(x,y)\in A}} 1=\frac{1}{\lvert S_{k} \rvert }\sum_{A\in\mathcal{A}_k} \ \sum_{\substack{y\in S_k \\(x,y)\in A}} 1.
\end{align*}
Recall the notation $A'=A^{\{1,\ldots,k-1\}}$.  
Since for each $x\in Q_{k-1}$ and $A\in\mathcal{A}_k$, there exists a unique $y\in S_k$ with $(x,y)\in A$ if and only if $x\in A'$, then we have
\begin{align*}
\alpha_k(x)\le \frac{1}{\lvert S_{k} \rvert }\sum_{\substack{A\in\mathcal{A}_k\\x\in A'}} 1.
\end{align*}
We then deduce
\begin{align*}
\alpha_k(x)^2\le \frac{1}{\lvert S_{k} \rvert ^2}\sum_{\substack{A_1,A_2\in\mathcal{A}_k\\x\in A_1'\cap A_2'}} 1,
\end{align*}
so that
\begin{align*}
\sum_{x\in Q_{k-1}}w_{k-1}(x)\alpha_k(x)^2\le \frac{1}{\lvert S_{k} \rvert ^2}\sum_{A_1,A_2\in\mathcal{A}_k} \ \sum_{\substack{x\in Q_{k-1}\\x\in A_1'\cap A_2'}}w_{k-1}(x).
\end{align*}
Thus, we deduce
\begin{align*}
E_{k-1} \le \frac{1}{\lvert S_{k} \rvert ^2}\sum_{A_1,A_2\in\mathcal{A}_k}w_{k-1}(A_1'\cap A_2').
\end{align*}

If the intersection of $A_1'$ and $A_2'$ is empty, then $w_{k-1}(A_1'\cap A_2')=0$.
If the intersection of $A_1'$ and $A_2'$ is non-empty, then the intersection is a hyperplane with 
\[
(F(A_1)\setminus \{k\})\cup (F(A_2)\setminus \{k\}) = F_{0}(A_{1}) \cup F_{0}(A_{2})
\]
as its set of fixed coordinates.
Let $F_1=F_{0}(A_{1})$ and $F_2=F_{0}(A_{2})$.
Recall that $\Vert F_i \Vert > C_0/\lvert S_{k} \rvert $ for $i \in \{1,2\}$.
Note that $F_1$ and $F_2$ uniquely determine $A_1$ and $A_2$ in $\mathcal A_k$,  respectively, since no two hyperplanes in $\mathcal{A}$ are parallel.
From \eqref{couldbealemma}, we obtain
\begin{align*}
E_{k-1} &\le \frac{1}{\lvert S_{k} \rvert ^2}\sum_{A_1,A_2\in\mathcal{A}_k}\nu(F(A_1'\cap A_2') )\\
&= \frac{1}{\lvert S_{k} \rvert ^2}\sum_{A_1,A_2\in\mathcal{A}_k}\nu(F_{0}(A_{1})\cup F_{0}(A_{2})) \\
&\le \frac{1}{\lvert S_{k} \rvert ^2}\sum_{\substack{F_1,F_2\subseteq \{1,\ldots,k-1\} \\ \Vert F_1 \Vert>C_0/\lvert S_{k} \rvert ,\  \Vert F_2 \Vert>C_0/\lvert S_{k} \rvert }} \nu(F_1\cup F_2) \\
&= \frac{1}{\lvert S_{k} \rvert ^2} \sum_{\substack{F_1,F_2\subseteq \{1,\ldots,k-1\} \\ \Vert F_1 \Vert>C_0/\lvert S_{k} \rvert ,\  \Vert F_2 \Vert>C_0/\lvert S_{k} \rvert }} \ \prod_{j\in F_1\cup F_2} \frac{1}{(1-\delta_j)\lvert S_{j} \rvert },
\end{align*}
finishing the proof.
\end{proof}

As a consequence of Lemma~\ref{sflemma2} and Lemma~\ref{sflemma3}, we immediately obtain the following.

\begin{corollary}\label{corsflemma3}
Fix a constant $C_0\geq 0$.
Let $\mathcal{A}$ be a collection of hyperplanes, pairwise non-parallel, in $Q=S_1\times\ldots\times S_n$ satisfying \eqref{normbound}.
Then, for each integer $k \in \{ 1, 2, \ldots, n \}$, we have 
\begin{align*}
w_{k}(B_k) \le \frac{1}{4\delta_k(1-\delta_k) \lvert S_{k} \rvert ^2} \sum_{\substack{F_1,F_2\subseteq \{1,\ldots,k-1\} \\ \Vert F_1 \Vert>C_0/\lvert S_{k} \rvert ,\  \Vert F_2 \Vert>C_0/\lvert S_{k} \rvert }} \ \prod_{j\in F_1\cup F_2} \frac{1}{(1-\delta_j)\lvert S_{j} \rvert }.
\end{align*}
\end{corollary}

We also indicate a different way to express the same bound on $w_{k}(B_k)$ which leads however to easier computations.  

\begin{corollary}\label{cor3pt5improved}
Fix a constant $C_{0} \ge 0$. 
Let $\mathcal A$ be a collection of hyperplanes, pairwise non-parallel, in $Q = S_{1} \times \cdots \times S_{n}$ such that
for every hyperplane $A \in \mathcal A$ we have $\Vert F(A) \Vert > C_{0}$.  Fix $k \in \{ 1, 2, \ldots, n \}$.  
Let $r$ be the minimal positive integer such that $\lvert S_{t} \rvert  > C_0/\lvert S_{k} \rvert $ for all $t \ge r$,
and suppose $r \le k-1$.  Define
\[
U = \sum_{\substack{F_{1} \subseteq \{ 1, \ldots, r-1 \} \\ \Vert F_{1} \Vert \le C_{0}/\lvert S_{k} \rvert }} 
\sum_{\substack{F_{2} \subseteq \{ 1, \ldots, r-1 \} \\ \Vert F_{2} \Vert \le C_{0}/\lvert S_{k} \rvert }} 
\prod_{j \in F_{1} \cup F_{2}} \dfrac{1}{(1-\delta_{j}) \lvert S_{j} \rvert } 
\]
and
\[
V = \sum_{\substack{F_{1} \subseteq \{ 1, \ldots, r-1 \} \\ \Vert F_{1} \Vert \le C_{0}/\lvert S_{k} \rvert }} 
\sum_{\substack{F_{2} \subseteq \{ 1, \ldots, r-1 \} \\ \Vert F_{2} \Vert > C_{0}/\lvert S_{k} \rvert }} 
\prod_{j \in F_{1} \cup F_{2}} \dfrac{1}{(1-\delta_{j}) \lvert S_{j} \rvert }.
\]
Then 
\begin{align*}
w_{k}(B_k) 
&\le  \dfrac{1}{4 \delta_{k} (1-\delta_{k}) \lvert S_{k} \rvert ^{2}} \Bigg( \prod_{j=1}^{k-1} \bigg(  1 + \dfrac{3}{(1-\delta_{j}) \lvert S_{j} \rvert }  \bigg) \\[5pt]
&\qquad \qquad -  2 \,(U+V) \prod_{j=r}^{k-1} \bigg(  1 + \dfrac{1}{(1-\delta_{j}) \lvert S_{j} \rvert }  \bigg) + U \Bigg).
\end{align*}
\end{corollary}

Before going to the proof, we clarify the interpretation of $U$ and $V$ when $r = 1$.  
With $r = 1$, the double sum in the definition of $U$ has either zero terms or exactly one term corresponding to $F_{1} = F_{2} = \emptyset$.
Since $\Vert \emptyset \Vert = 1$, the term exists precisely when $C_{0}/\lvert S_{k} \rvert \ge 1$.  
The empty product is $1$ so that in this case $U = 1$. 
On the other hand, if $C_{0}/\lvert S_{k} \rvert < 1$, then $\emptyset$ does not satisfy the conditions on $F_{1}$ and $F_{2}$ in the double sum;
hence, in this case, the double sum has no terms and is $0$.  
Since we cannot have both $\Vert \emptyset \Vert \le C_{0}/\lvert S_{k} \rvert$ and  $\Vert \emptyset \Vert > C_{0}/\lvert S_{k} \rvert$, 
we deduce $V = 0$ whenever $r = 1$. 

\begin{proof}[Proof of Corollary~\ref{cor3pt5improved}]
From Corollary \ref{corsflemma3}, we obtain
\begin{align*}
   4\delta_k(1-\delta_k) \lvert S_{k} &\rvert ^2 w_{k}(B_k) \le \sum_{\substack{F_1,F_2\subseteq \{1,\ldots,k-1\} \\ \Vert F_1 \Vert>C_0/\lvert S_{k} \rvert ,\  \Vert F_2 \Vert>C_0/\lvert S_{k} \rvert }} \ \prod_{j\in F_1\cup F_2} \frac{1}{(1-\delta_j)\lvert S_{j} \rvert } \\[5pt]
    &= \sum_{\substack{J\subseteq\{1,\ldots,k-1\}}} \  \sum\limits_{\substack{F_1,F_2\subseteq\{1,\ldots,k-1\} \\ F_1\cup F_2=J}}  \ \prod_{j\in J} \frac{1}{(1-\delta_j)\lvert S_{j} \rvert }\\[5pt]
    &\qquad\qquad - \sum_{\substack{F_{1} \subseteq \{ 1, \ldots, k-1 \} \\ \Vert F_{1} \Vert \le C_{0}/\lvert S_{k} \rvert }} \sum_{\substack{F_{2} \subseteq \{ 1, \ldots, k-1 \}}} \prod_{j \in F_{1} \cup F_{2}} \dfrac{1}{(1-\delta_{j}) \lvert S_{j} \rvert }\\[5pt]
    &\qquad\qquad - \sum_{\substack{F_{2} \subseteq \{ 1, \ldots, k-1 \} \\ \Vert F_{2} \Vert \le C_{0}/\lvert S_{k} \rvert }} \sum_{\substack{F_{1} \subseteq \{ 1, \ldots, k-1 \}}} \prod_{j \in F_{1} \cup F_{2}} \dfrac{1}{(1-\delta_{j}) \lvert S_{j} \rvert }\\[5pt]
    &\qquad\qquad + \sum_{\substack{F_{1} \subseteq \{ 1, \ldots, k-1 \} \\ \Vert F_{1} \Vert \le C_{0}/\lvert S_{k} \rvert }} \sum_{\substack{F_{2} \subseteq \{ 1, \ldots, k-1 \} \\ \Vert F_{2} \Vert \le C_{0}/\lvert S_{k} \rvert }} \prod_{j \in F_{1} \cup F_{2}} \dfrac{1}{(1-\delta_{j}) \lvert S_{j} \rvert }.
\end{align*}
    For $i \in \{1,2\}$, if $\Vert F_i \Vert \le C_0/\lvert S_{k} \rvert $, then $F_i \subseteq \{1, 2, \ldots, r-1\}$, which we obtain from the definition of $\Vert F_i \Vert$ and $r$.
Hence, we deduce
\begin{align*}
    4\delta_k(1-\delta_k) \lvert S_{k} \rvert ^2 w_{k}(B_k) &\le \sum_{\substack{J\subseteq\{1,\ldots,k-1\}}} \  \sum\limits_{\substack{F_1,F_2\subseteq\{1,\ldots,k-1\} \\ F_1\cup F_2=J}}  \ \prod_{j\in J} \frac{1}{(1-\delta_j)\lvert S_{j} \rvert }\\[5pt]
    &\qquad - \sum_{\substack{F_{1} \subseteq \{ 1, \ldots, r-1 \} \\ \Vert F_{1} \Vert \le C_{0}/\lvert S_{k} \rvert }} \sum_{\substack{F_{2} \subseteq \{ 1, \ldots, k-1 \}}} \prod_{j \in F_{1} \cup F_{2}} \dfrac{1}{(1-\delta_{j}) \lvert S_{j} \rvert }\\[5pt]
    &\qquad - \sum_{\substack{F_{2} \subseteq \{ 1, \ldots, r-1 \} \\ \Vert F_{2} \Vert \le C_{0}/\lvert S_{k} \rvert }} \sum_{\substack{F_{1} \subseteq \{ 1, \ldots, k-1 \}}} \prod_{j \in F_{1} \cup F_{2}} \dfrac{1}{(1-\delta_{j}) \lvert S_{j} \rvert }\\[5pt]
    &\qquad + \sum_{\substack{F_{1} \subseteq \{ 1, \ldots, r-1 \} \\ \Vert F_{1} \Vert \le C_{0}/\lvert S_{k} \rvert }} \sum_{\substack{F_{2} \subseteq \{ 1, \ldots, r-1 \} \\ \Vert F_{2} \Vert \le C_{0}/\lvert S_{k} \rvert }} \prod_{j \in F_{1} \cup F_{2}} \dfrac{1}{(1-\delta_{j}) \lvert S_{j} \rvert }.
\end{align*}
The last double sum of a product above is equal to $U$. 
Considering the second double sum of a product on the right-hand side of the above inequality, we can express $F_2$ as $A\cup B$ where $A\subseteq \{1,\ldots,r-1\}$ and $B\subseteq \{r,\ldots,k-1\}$, so we obtain
\begin{align*}
   &\sum_{\substack{F_{1} \subseteq \{ 1, \ldots, r-1 \} \\ \Vert F_{1} \Vert \le C_{0}/\lvert S_{k} \rvert }} \, \sum_{\substack{F_{2} \subseteq \{ 1, \ldots, k-1 \}}} \, \prod_{j \in F_{1} \cup F_{2}} \dfrac{1}{(1-\delta_{j}) \lvert S_{j} \rvert }\\[5pt]
    &\quad =\sum_{\substack{F_{1} \subseteq \{ 1, \ldots, r-1 \} \\ \Vert F_{1} \Vert \le C_{0}/\lvert S_{k} \rvert }} \sum_{\substack{A \subseteq \{ 1, \ldots, r-1 \}}}\sum_{\substack{B \subseteq \{ r, \ldots, k-1 \}}} \prod_{j \in F_{1} \cup (A\cup B)} \dfrac{1}{(1-\delta_{j}) \lvert S_{j} \rvert }\\[5pt]
    &\quad =\sum_{\substack{F_{1} \subseteq \{ 1, \ldots, r-1 \} \\ \Vert F_{1} \Vert \le C_{0}/\lvert S_{k} \rvert }} \sum_{\substack{A \subseteq \{ 1, \ldots, r-1 \}}}\sum_{\substack{B \subseteq \{ r, \ldots, k-1 \}}} \prod_{j \in F_{1} \cup A} \dfrac{1}{(1-\delta_{j}) \lvert S_{j} \rvert } \prod_{j \in B} \dfrac{1}{(1-\delta_{j}) \lvert S_{j} \rvert }\\[5pt]
    &\quad =\sum_{\substack{F_{1} \subseteq \{ 1, \ldots, r-1 \} \\ \Vert F_{1} \Vert \le C_{0}/\lvert S_{k} \rvert }} \sum_{\substack{A \subseteq \{ 1, \ldots, r-1 \}}}\prod_{j \in F_{1} \cup A} \dfrac{1}{(1-\delta_{j}) \lvert S_{j} \rvert }\sum_{\substack{B \subseteq \{ r, \ldots, k-1 \}}} \prod_{j \in B} \dfrac{1}{(1-\delta_{j}) \lvert S_{j} \rvert },
\end{align*}
where the second to last equality holds since $(F_1\cup A)\cap B=\emptyset$.
Observe that 
\begin{align*}
\sum_{\substack{B \subseteq \{ r, \ldots, k-1 \}}} \prod_{j \in B} \dfrac{1}{(1-\delta_{j}) \lvert S_{j} \rvert }=\prod_{j=r}^{k-1}\left(1+\frac{1}{(1-\delta_j)\lvert S_{j} \rvert }\right).
\end{align*}
We also have
\begin{align*}
\sum_{\substack{F_{1} \subseteq \{ 1, \ldots, r-1 \} \\ \Vert F_{1} \Vert \le C_{0}/\lvert S_{k} \rvert }} &\, \sum_{\substack{A \subseteq \{ 1, \ldots, r-1 \}}} \, \prod_{j \in F_{1} \cup A} \dfrac{1}{(1-\delta_{j}) \lvert S_{j} \rvert }\\[5pt]
&=\sum_{\substack{F_{1} \subseteq \{ 1, \ldots, r-1 \} \\ \Vert F_{1} \Vert \le C_{0}/\lvert S_{k} \rvert }} \sum_{\substack{A \subseteq \{ 1, \ldots, r-1 \}\\ \Vert A \Vert > C_0/\lvert S_{k} \rvert }}\prod_{j \in F_{1} \cup A} \dfrac{1}{(1-\delta_{j}) \lvert S_{j} \rvert }\\[5pt]
&\quad \qquad+\sum_{\substack{F_{1} \subseteq \{ 1, \ldots, r-1 \} \\ \Vert F_{1} \Vert \le C_{0}/\lvert S_{k} \rvert }} \sum_{\substack{A \subseteq \{ 1, \ldots, r-1 \}\\ \Vert A \Vert \le C_0/\lvert S_{k} \rvert }}\prod_{j \in F_{1} \cup A} \dfrac{1}{(1-\delta_{j}) \lvert S_{j} \rvert }\\[5pt]
&=U+V.
\end{align*}
Thus, we deduce
\begin{align*}
    \sum_{\substack{F_{1} \subseteq \{ 1, \ldots, r-1 \} \\ \Vert F_{1} \Vert \le C_{0}/\lvert S_{k} \rvert }} \, \sum_{\substack{F_{2} \subseteq \{ 1, \ldots, k-1 \}}} \, &\prod_{j \in F_{1} \cup F_{2}} \dfrac{1}{(1-\delta_{j}) \lvert S_{j} \rvert } \\
    &\quad =(U+V)\prod_{j=r}^{k-1}\left(1+\frac{1}{(1-\delta_j)\lvert S_{j} \rvert }\right).
\end{align*}
Observe that this last equation is equivalent to
\begin{align*}
    \sum_{\substack{F_{2} \subseteq \{ 1, \ldots, r-1 \} \\ \Vert F_{2} \Vert \le C_{0}/\lvert S_{k} \rvert }} \sum_{\substack{F_{1} \subseteq \{ 1, \ldots, k-1 \}}} &\prod_{j \in F_{1} \cup F_{2}} \dfrac{1}{(1-\delta_{j}) \lvert S_{j} \rvert } \\
    &\quad =(U+V)\prod_{j=r}^{k-1}\left(1+\frac{1}{(1-\delta_j)\lvert S_{j} \rvert }\right).
\end{align*}
Rearranging the order of our second sum and product below, we have
\begin{align*}
\sum_{J\subseteq\{1,\ldots,k-1\}} \ & \sum\limits_{\substack{F_1,F_2\subseteq\{1,\ldots,k-1\} \\ F_1\cup F_2=J}} \prod_{j\in J}  \frac{1}{(1-\delta_{j})\lvert S_{j} \rvert } \\
&\qquad =\sum_{J\subseteq\{1,\ldots,k-1\}} \ \prod_{j\in J} \frac{1}{(1-\delta_{j})\lvert S_{j} \rvert } \sum\limits_{\substack{F_1,F_2\subseteq\{1,\ldots,k-1\} \\ F_1\cup F_2=J}} 1.
\end{align*}
Since each element of $J$ with $F_1\cup F_2=J$ is either in $F_1$ and not $F_2$, in $F_2$ and not $F_1$, or in both $F_1$ and $F_2$, we deduce 
\begin{align*}
\sum\limits_{\substack{F_1,F_2\subseteq\{1,\ldots,k-1\} \\ F_1\cup F_2=J}} 1=3^{\lvert J \rvert }.
\end{align*}
Substituting, we obtain
\begin{align*}
\sum_{\substack{J\subseteq\{1,\ldots,k-1\}}} &\ \sum\limits_{\substack{F_1,F_2\subseteq\{1,\ldots,k-1\} \\ F_1\cup F_2=J}}  \ \prod_{j\in J} \frac{1}{(1-\delta_j)\lvert S_{j} \rvert } \\
&\qquad = \sum_{J\subseteq\{1,\ldots,k-1\}} 3^{\lvert J \rvert } \ \prod_{j\in J}  \frac{1}{(1-\delta_{j})\lvert S_{j} \rvert } \\
&\qquad =\sum_{J\subseteq\{1,\ldots,k-1\}} \ \prod_{j\in J} \frac{3}{(1-\delta_{j})\lvert S_{j} \rvert } \\
&\qquad =\prod_{j=1}^{k-1} \bigg(  1 + \dfrac{3}{(1-\delta_{j}) \lvert S_{j} \rvert }  \bigg).
\end{align*}
Combining the above, we conclude that
\begin{align*}
   w_{k}(B_k) 
&\le  \dfrac{1}{4 \delta_{k} (1-\delta_{k}) \lvert S_{k} \rvert ^{2}} \Bigg( \prod_{j=1}^{k-1} \bigg(  1 + \dfrac{3}{(1-\delta_{j}) \lvert S_{j} \rvert }  \bigg) \\[5pt]
&\qquad \qquad -  
2 \,(U+V) \prod_{j=r}^{k-1} \bigg(  1 + \dfrac{1}{(1-\delta_{j}) \lvert S_{j} \rvert }  \bigg) 
+ U \Bigg),
\end{align*}
which completes the proof.
\end{proof}


The idea is to apply the prior upper bounds for $w_{k}(B_k)$ to estimate the value of $w_{k}(B_k)$ for $k \le N$, where in the end we will take $N = 10^{6}$.  
Next, we show how to find an upper bound for $w_{k}(B_k)$ for $k > N$ and then find an upper bound for
\[
\sum_{k > N} w_{k}(B_k).
\]
For this part we require $N \ge 61$ to be an integer and $k > N$.  
Note that we view $N$ as fixed, so we will allow constants below to depend on $N$.
We set $\delta_{j} = 1/2$ for all $j > N$.
As we will be using Corollary~\ref{cor3pt5improved} to compute $w_{k}(B_k)$ for $k = N$, 
we will have already completed most of the calculation for 
\begin{equation}\label{mzerodef}
M_{0} = \prod_{j=1}^{N} \bigg(  1 + \dfrac{3}{(1-\delta_{j}) \lvert S_{j} \rvert }  \bigg),
\end{equation}
so we make use of it.  Finally, we denote the $j^{\text{th}}$ prime by $p_{j}$ and the number of primes $\le x$ by $\pi(x)$.  

\begin{lemma}\label{tailendtheorem}
With the above notation, we set
\[
c_1 = -\log \log  p_{N} + \frac{1}{\log ^2 p_{N}} 
\quad \text{ and } \quad
c_{2} = 1 + \dfrac{3}{2\log p_{N}}.
\]
If $\lvert S_{j} \rvert  = p_{j}$ for every $j > N$, then
\begin{align*}
\sum_{k > N} w_{k}(B_k) &\le\dfrac{2 c_{2} M_{0} e^{6c_{1}} }{p_{N}}\cdot\Big( \log^{5} p_{N}+ 5 \log^{4} p_{N}  + 20 \log^{3} p_{N}  \\[5pt]
&\qquad \qquad + 60 \log^{2} p_{N} + 120 \log p_{N} + 120 \Big).
\end{align*}
\end{lemma}

\begin{proof}
From $\delta_{k} = 1/2$ and Corollary~\ref{cor3pt5improved} (with $C_{0} = 0$ so $r = 1$ and $U = V = 0$),
we see that, for $k > N$, we have
\begin{align}\label{tailendeq1}
\begin{aligned}
w_{k}(B_k) 
&\le  \dfrac{1}{4 \delta_{k} (1-\delta_{k}) \lvert S_{k} \rvert ^{2}}  \prod_{j=1}^{k-1} \bigg(  1 + \dfrac{3}{(1-\delta_{j}) \lvert S_{j} \rvert }  \bigg) \\
&=  \dfrac{1}{\lvert S_{k} \rvert ^{2}} \prod_{j=1}^{k-1} \bigg(  1 + \dfrac{3}{(1-\delta_{j}) \lvert S_{j} \rvert }  \bigg).
\end{aligned}
\end{align}
Since $k > N$ and $\delta_{j} = 1/2$ for all $j > N$, we obtain
\begin{equation}\label{tailendeq2}
\prod_{j=1}^{k-1} \bigg( 1 + \frac{3}{(1-\delta_j)\lvert S_{j} \rvert } \bigg) =M_{0} 
 \prod_{j=N+1}^{k-1} \bigg( 1 + \frac{6}{\lvert S_{j} \rvert } \bigg) \leq M_0 \exp \bigg( 6 \sum_{j=N+1}^{k-1} \frac{1}{\lvert S_{j} \rvert } \bigg),
\end{equation}
 where we have used that $1+x \le e^x$ for all real numbers $x$ (the function $e^x$ is convex up and $y=1+x$ is a tangent line to its graph at $x=0$).

We are now ready to make use of the specification that $\lvert S_{j} \rvert  = p_{j}$ for every $j > N$.  
From the work of J.~B.~Rosser and L.~Schoenfeld \cite[Theorem~5]{rossschon}, we have the estimates
\[
\log \log x + B - \frac{1}{2 \log ^2 x} \leq \sum_{p \leq x} \frac{1}{p} < \log \log x + B + \frac{1}{2 \log ^2 x},
\qquad \text{ for } x \geq 286,
\]
for some constant $B \approx 0.2614972128$.  
As $N+1 \ge 62$ and $p_{62} = 293 > 286$, we deduce that
\begin{align*}
\sum_{j=N+1}^{k-1} &\frac{1}{\lvert S_{j} \rvert } 
= \sum_{j=N+1}^{k-1} \frac{1}{p_{j}} 
= \sum_{p \le p_{k-1}} \frac{1}{p} - \sum_{p \le p_{N}} \frac{1}{p} \\[5pt]
&\quad < \bigg(  \log \log p_{k-1} + B + \frac{1}{2 \log^2 p_{k-1}}  \bigg) - \bigg(  \log \log p_{N} + B - \frac{1}{2 \log^2 p_{N}}  \bigg) \\[5pt]
&\quad = \log \log p_{k-1} - \log \log p_{N} + \frac{1}{2 \log^2 p_{k-1}}  + \frac{1}{2 \log^2 p_{N}}  \\[5pt]
&\quad \le \log \log p_{k-1} - \log \log p_{N} + \frac{1}{2 \log^2 p_{N}}  + \frac{1}{2 \log^2 p_{N}} 
= \log \log p_{k-1} +  c_1.
\end{align*}
From \eqref{tailendeq2}, we now see that
\[
\prod_{j=1}^{k-1} \bigg( 1 + \frac{3}{(1-\delta_j)\lvert S_{j} \rvert } \bigg)
\leq M_{0} \exp\big(  6  \log \log p_{k-1} +  6 c_1  \big)
= M_0 e^{6c_{1}} \log^{6} p_{k-1}.
\]
From \eqref{tailendeq1}, we obtain the estimate for $w_{k}(B_k)$ for $k > N$ that we will want, namely
\[
w_{k}(B_k) \le M_0 e^{6c_{1}} \dfrac{\log^{6} p_{k}}{p_{k}^{2}}.
\]

Next, we want an estimate of the sum over $k > N$ of this bound for $w_{k}(B_k)$. 
We make use of a Riemann-Stieltjes integral to obtain
\begin{align*}
\sum_{k = N+1}^{\infty}  \dfrac{\log^{6} p_{k}}{p_{k}^{2}}
&\le  \int_{p_{N}}^{\infty} \frac{\log^{6} t}{t^2} \,d\,\pi (t) \\[5pt]
&= \frac{\pi(t)\log^{6} t}{t^2}\bigg\vert_{p_{N}}^{\infty} - \int_{p_{N}}^{\infty} \pi (t) \,d \bigg( \frac{\log^{6} t}{t^2} \bigg) \\[5pt]
&\le 2 \int_{p_{N}}^{\infty} \dfrac{\pi(t) \log^{6} t}{t^{3}} \,dt,
\end{align*}
where we have used that
\[
d \bigg( \frac{\log^{6} t}{t^2} \bigg) = \bigg(  \dfrac{6 \log^{5} t}{t^{3}} - \dfrac{2 \log^{6} t}{t^{3}}  \bigg) \,dt
\]
and ignored negative quantities.
From J.~B.~Rosser and L.~Schoenfeld \cite[Theorem~1]{rossschon}, we have
\[
\pi(x) < \dfrac{x}{\log x} \bigg(  1 + \dfrac{3}{2\log x}  \bigg) \qquad \text{ for all $x > 1$}.
\]
Thus, for $t \ge p_{N}$, we obtain $\pi(t) \le c_{2} \,t/\log t$.
Thus,
\[
\sum_{k = N+1}^{\infty}  \dfrac{\log^{6} p_{k}}{p_{k}^{2}} \le 2 c_{2} \int_{p_{N}}^{\infty} \dfrac{\log^{5} t}{t^{2}} \,dt.
\]
The latter integral can be computed exactly to obtain
\begin{align*}
\sum_{k = N+1}^{\infty}  \dfrac{\log^{6} p_{k}}{p_{k}^{2}} 
&\le \dfrac{2 c_{2}}{p_{N}} \Big( \log^{5} p_{N} + 5 \log^{4} p_{N} + 20 \log^{3} p_{N}\\[5pt]
&\qquad + 60 \log^{2} p_{N} + 120 \log p_{N} + 120 \Big).
\end{align*}
Combining the above, the lemma follows.
\end{proof}

\subsection{Proof of Theorem~\ref{thm118}}

We will now prove Theorem \ref{thm118}, which states that every covering system with distinct squarefree moduli has a minimum modulus which is $\le 118$.

\begin{proof}
Let $\mathcal{A}$ be a collection of hyperplanes covering $Q=S_1\times\ldots\times S_n$ corresponding to a covering system with distinct squarefree moduli. 
Recall that $\lvert S_{k} \rvert = p_{k}$ and if $A$ is a hyperplane corresponding to some congruence with squarefree modulus $m$ in a covering system, then $\Vert F(A) \Vert = m$.
For the sake of contradiction, assume that $\Vert F(A) \Vert > C_0=118$  for all $A \in \mathcal A$. For our computations, we used Maple 2019.

We choose $\delta_j$ as below:
\begin{gather*}
\delta_{1} = \cdots = \delta_{7} = 0, \quad \delta_{8} = 0.171, \quad 
\delta_{9} = 0.190, \quad \delta_{10} = 0.199, \\[4pt]
\delta_{11} = 0.210, \quad \delta_{12} = 0.210, \quad \delta_{13} = 0.224, \quad \delta_{14} = 0.233,  \\[4pt]
\delta_{15} = 0.237, \quad \delta_{16} = 0.237, \quad \delta_{17} = 0.237, \quad \delta_{18} = 0.252,  \\[4pt]
\delta_{19} = 0.252, \quad \delta_{20} = 0.255, \quad \delta_{21} = 0.260, \quad \delta_{22} = 0.261,  \\[4pt]
\quad \delta_{23} = 0.263, \quad \delta_{24} = 0.264,\quad \delta_{25} = 0.262,  \quad \delta_{26} = 0.265,  \quad \delta_{27} = 0.269,\\[4pt]
\delta_{j} = 0.279 \ \ \text{(for $28 \le j \le 35$)}, \qquad
\delta_{j} = 0.289 \ \ \text{(for $36 \le j \le 45$)},  \\[4pt]
\delta_{j} = 0.297 \ \ \text{(for $46 \le j \le 60$)}, \qquad
\delta_{j} = 0.307 \ \ \text{(for $61 \le j \le 99$)},  \\[4pt]
\delta_{j} =0.331 \ \ \text{(for $100 \le j \le 1000$)}, \qquad
\delta_{j} = 0.372 \ \ \text{(for $1001 \le j \le 10000$)},  \\[4pt]
\delta_{j} = 0.418 \ \ \text{(for $10001 \le j \le 1000000$)}, \qquad \delta_{j}=0.5 \ \ \text{(for $j \ge 1000001$)}.
\end{gather*}

Using Corollary \ref{corsflemma1}, we compute that $w_1(B_1)=w_2(B_2)=w_3(B_3)=0$, $w_4(B_4) \leq 1/210$, $w_5(B_5) \leq 3/110$, $w_6(B_6) \leq 50/1001$, and $w_7(B_7) \leq 43/715$, so we have
\begin{align}\label{1-7}
    \sum_{k=1}^7 w_k(B_k) \leq \frac{194}{1365} = 0.142124542124\ldots.
\end{align}
Using Corollary \ref{cor3pt5improved} and taking into account the comments before its proof, we calculate
\begin{align}\label{8-106}
    \sum_{k=8}^{10^6} w_k(B_k) \leq 0.856857558639\ldots.
\end{align}
The computations for \eqref{8-106}, in particular for the upper bound on $w_{N}(B_{N})$ where $N = 10^{6}$, 
provide us with all but the last factor (where $j = N$) of the product for $M_{0}$ in \eqref{mzerodef}.  
Including that factor and applying Lemma \ref{tailendtheorem}, we obtain
\begin{align}\label{tail}
    \sum_{k>10^6} w_k(B_k) \leq 0.000402960685\ldots.
\end{align}
Combining (\ref{1-7}), (\ref{8-106}), and (\ref{tail}), we obtain
\begin{align}
    \sum_{k=1}^{\infty} w_k(B_k)\leq 0.999385061449\ldots < 1.
\end{align}
Thus, by Lemma \ref{newlemma1}, $\mathcal{A}$ does not cover $Q$, which is a contradiction.
Therefore, every covering system with distinct squarefree moduli has a minimum modulus which is $\le 118$.
\end{proof}

\section{Proof of Theorem~\ref{secondresult}}

We proceed by induction on $k$, with the case $k = 0$ being initially covered in \cite{ref9}.
Suppose that we know the result holds for $k-1$, and assume that it does not hold for $k$.
In other words, we assume we know that $m_{1}, m_{2}, \dots, m_{k}$ are necessarily bounded 
in a covering system as in the statement of the theorem but that $m_{k+1}$ can be arbitrarily large.
Note that by the conditions in the theorem, the congruences modulo $m_{1}, m_{2}, \dots, m_{k}$
do not by themselves form a covering system (that is, at least one more congruence is needed).  
Since there are a finite number of choices for the minimal $k$ moduli in a covering system
and an infinite number of choices for the $(k+1)$st smallest modulus, there is some fixed choice
of $m_{1}, m_{2}, \dots, m_{k}$ for which there exist an infinite number of covering systems
$C_{1}, C_{2}, \dots$ satisfying:
\vskip -5pt
\centerline{\parbox[t]{11.2cm}{\begin{enumerate}
\item[(i)]
For each $i \ge 1$, the smallest $k$ moduli appearing in congruences in $C_{i}$ are $m_{1}, m_{2}, \dots, m_{k}$.
\item[(ii)]
For each $i \ge 1$, the congruences modulo $m_{1}, m_{2}, \dots, m_{k}$ in $C_{i}$ do not by themselves form a covering system.
\item[(iii)]
Let $m_{t}(C_{i})$ be the $t^{\rm th}$ smallest modulus of $C_{i}$ and $M(C_{i})$ the maximal
modulus of $C_{i}$.  Then $m_{k+1}(C_{1}) > B(1)$ and
\[
m_{k+1}(C_{i}) > \max\{ B(1), M(C_{i-1}) \} \quad \text{ for each } i \ge 2.
\]
\end{enumerate}}}

In (iii), the condition $m_{k+1}(C_{1}) > B(1)$ implies $M(C_{1}) > B(1)$ and, consequently, $\max\{ B(1), M(C_{i-1}) \} = M(C_{i-1})$ for each $i \ge 2$.  
Nevertheless, the inequality as written in (iii) serves the purpose of emphasizing the information we want to use in our argument.

The idea is to obtain a contradiction by constructing a covering of the integers using congruences
with distinct moduli all $> B(1)$, which will contradict the definition of $B(1)$ (and \cite{ref9}).
We will do this by covering one residue class modulo $m_{1} m_{2} \cdots m_{k}$ for each covering system
$C_{1}, C_{2}, \ldots, C_{m_{1}m_{2} \cdots m_{k}}$.  More precisely, for each $i$, we will show that 
the congruences in $C_{i}$ with moduli $> m_{k}$ can be used to cover one residue class modulo $m_{1} m_{2} \cdots m_{k}$.
Given (iii) above, the set of moduli $> m_{k}$ in $C_{i}$ are disjoint for different $i$, so that the congruences used as
$i$ varies involve distinct moduli.

Fix $i$, and write the congruences in $C_{i}$ as in \eqref{congrs}.  
Note that the $a_{j}$ and $m_{j}$ appearing there depend on $i$, 
but we will suppress that dependence here noting again that $i$ is fixed.
Let $b$ be a fixed arbitrary integer, and suppose we wish to find $r - k$ congruences with moduli
$m_{k+1}, \ldots, m_{r}$ such that every integer which is $b$ modulo $m_{1} m_{2} \cdots m_{k}$ 
satisfies at least one of the $r-k$ congruences.

Each congruence $x \equiv a_{j} \pmod{m_{j}}$ in \eqref{congrs} restricted to $1 \le j \le k$ 
is equivalent to $m_{1} m_{2} \cdots m_{k}/m_{j}$ congruences modulo $m_{1} m_{2} \cdots m_{k}$.
In other words, each congruence $x \equiv a_{j} \pmod{m_{j}}$, with $1 \le j \le k$, covers precisely $m_{1} m_{2} \cdots m_{k}/m_{j}$
residue classes of integers modulo $m_{1} m_{2} \cdots m_{k}$.
By (ii) above, the first $k$ congruences in \eqref{congrs} do not form a covering system,
so there is an $a \in \mathbb Z$ for which no integer satisfying
\begin{equation}\label{cong}
x \equiv a \pmod{m_{1} m_{2} \dots m_{k}}
\end{equation}
satisfies one of the first $k$ congruences in \eqref{congrs}.  
Each integer satisfying \eqref{cong} therefore satisfies at least one of the congruences
\[
x \equiv a_{j} \pmod{m_{j}}, \quad \text{ where } k+1 \le j \le r.
\]
We claim that each integer which is $b$ modulo $m_{1} m_{2} \dots m_{k}$
necessarily satisfies at least one of the congruences
\[
x \equiv a_{j} - a + b \pmod{m_{j}}, \quad \text{ where } k+1 \le j \le r.
\]
Indeed, for each integer $t$, we know that 
$a + t m_{1} m_{2} \dots m_{k}$ satisfies \eqref{cong}, so that 
\[
a + t m_{1} m_{2} \dots m_{k} \equiv a_{j} \pmod{m_{j}}, \quad \text{ for some } k+1 \le j \le r.
\]
By rewriting, for each $t \in \mathbb Z$, we get
\[
b + t m_{1} m_{2} \dots m_{k} \equiv a_{j} - a + b \pmod{m_{j}}, \quad \text{ for some } k+1 \le j \le r.
\]
This implies what was claimed.

Thus, for each $i$, we can choose a residue class modulo $m_{1} m_{2} \dots m_{k}$ and 
cover the integers in that residue class using congruences with the moduli from $C_{i}$ which are
$> m_{k}$.  We can therefore cover all
the residue classes modulo $m_{1} m_{2} \dots m_{k}$ by using the moduli $> m_{k}$ from the 
congruences in
\[
C_{1}, C_{2}, \dots, C_{m_{1} m_{2} \dots m_{k}}.
\]
As noted earlier, we deduce from (iii) above that these moduli are all distinct and $> B(1)$, contradicting
the case $k = 0$ established in \cite{ref9} and finishing the proof.



\bibliography{sn-bibliography}


\end{document}